\documentclass[numbers,webpdf,imanum]{ima-authoring-template}

\usepackage{amscd,verbatim}
\usepackage{upgreek,comment}
\usepackage{mystylefile_IMA}
\graphicspath{{Fig/}}

\theoremstyle{thmstyletwo}%
\newtheorem{theorem}{Theorem}[section]
\newtheorem{remark}{Remark}[section]

\numberwithin{equation}{section}

\begin{document}

\DOI{DOI HERE}
\copyrightyear{}
\vol{00}
\pubyear{}
\access{Advance Access Publication Date: Day Month Year}
\appnotes{Paper}
\copyrightstatement{Published by Oxford University Press on behalf of the Institute of Mathematics and its Applications. All rights reserved.}
\firstpage{1}

\title[Randomized Normal Equations]{Solution of Least Squares Problems\\ with Randomized Preconditioned Normal Equations}

\author{Ilse C.F. Ipsen\footnote{The work was supported in part by NSF grant DMS-1760374, NSF grant CCF-2209510, and DOE grant DE-SC0022085.}
\address{\orgdiv{Department of Mathematics}, \orgname{North Carolina State University}, 
\orgaddress{ \state{NC} \postcode{}27695-8205, \country{USA}}}}

\authormark{Ilse C.F. Ipsen}


\abstract{We consider the solution of full column-rank least squares problems by means of 
normal equations that are preconditioned, symmetrically or non-symmetrically, with a randomized
preconditioner. We present non-intuitive but realistic perturbation bounds for the relative error in the computed solutions
and illustrate that, with an effective preconditioner, these bounds accurately adapt to the size of
the least squares residual. This means,
 the accuracy of the preconditioned normal equations depends on the residual
of the original least squares problem --even though there is officially no least squares problem present.
 Probabilistic condition number bounds corroborate the effectiveness
of the randomized preconditioner computed with small amounts of sampling.
 Numerical experiments illustrate that the perturbation bounds are informative and that the preconditioned
normal equations, when computed with direct methods, can be almost as accurate as those from the QR-based Matlab backslash (mldivide) command -- even for highly illconditioned matrices. }



\maketitle
\bigskip

\section{Introduction}
Given a matrix $\ma\in\rmn$ with $\rank(\ma)=n$, and $\vb\in\real^m$ we consider
the solution of the least squares problem
\begin{equation}\label{e_ls}
\min_{\vx}{\|\ma\vx-\vb\|_2},
\end{equation}
which has the unique solution\footnote{The superscript $\dagger$ denotes the Moore Penrose inverse.}
 $\vx_*\equiv \ma^{\dagger}\vb$. The preferred direct solution
method is a QR or Singular Value Decomposition \cite[Chapter 5]{GovL13}.
Instead, we consider a method based on the normal equations.

The normal equations\footnote{The superscript $T$ denotes the transpose, and the two-norm condition number with respect to left 
inversion is $\kappa(\ma)\equiv \|\ma\|_2\|\ma^{\dagger}\|_2.$}
\begin{equation}\label{e_ne1}
\ma^T\ma\vx=\ma^T\vb
\end{equation}
are usually not recommended due to potential numerical instability
\cite[Section 5.3.7]{GovL13}. Since their condition number
is $\kappa(\ma^T\ma)=\kappa(\ma)^2$,
the normal equations are numerically singular in IEEE double precision 
 once $\kappa(\ma)\geq 10^7$.
 Our proposed remedy is to precondition the normal equations, either on both sides or only on the
 left. We derive perturbation bounds to justify that the preconditioned normal equations can 
 represent numerically
 stable algorithms for solving least squares problems.

\paragraph{Preconditioned Normal Equations.}
We precondition $\ma_p\equiv \ma\mr_s^{-1}$ with an 
effective randomized preconditioner $\mr_s$ so that the preconditioned matrix $\ma_p$ is well conditioned
with high probability, and solve the preconditioned normal equations
\begin{equation*}
\begin{split}
\ma_p^T\ma_p\vy&=\ma_p^T\vb\\
\mr_s\vx&=\vy.
\end{split}
\end{equation*}
We consider a preconditioner to be `effective', if the condition number of the preconditioned 
matrix does not exceed 10, i.e. $\kappa(\ma_p)\leq 10$.

\paragraph{Half Preconditioned Normal Equations.}
The alternative is to dispense with the triangular system solution  by solving
\begin{align*}
\ma_p^T\ma\vx=\ma_p^T\vb,
\end{align*}
which represents an 
instance of the \textit{not-normal equations} \cite[(2.3)]{Wathen2025}.
The matrix $\ma_p^T\ma$ is nonsymmetric, so the linear system 
would have to be solved by LU
with partial pivoting or a QR factorization \cite[Section 2]{Wathen2025}.
If an iterative solver were used, these would correspond to the left preconditioned
CGNE equations \cite[Section 11.3.9.]{GovL13}. 
We do not encounter the 'matrix
squaring problem' \cite[Section 2.1]{Wathen2022}, because we use a preconditioner
for $\ma$ rather than $\ma^T\ma$.

Figure~\ref{fig1} illustrates that,
with an effective preconditioner, the computed solutions for the 
preconditioned and  half-preconditioned normal equations are almost as accurate as the 
QR-based Matlab backslash solution. In particular, they accurately capture the increase in the
least squares residual. This means, the solution accuracy of the preconditioned
normal equations depends on the residual of 
the original 
least squares problem (\ref{e_ls}) --even though there is officially no least squares problem present
or there is not even a corresponding least squares problem.

\subsection{Contributions and Overview}
We present a perturbation theory to show that the normal equations, preconditioned with an effective randomized preconditioner
on one or both sides can be highly accurate, even for ill conditioned matrices.
We present non-intuitive but realistic perturbation bounds for the relative error in the computed solutions
of the preconditioned and half-preconditioned normal equations, and show the following:
\begin{enumerate}
\item Our non-symmetric perturbation bounds for the preconditioned normal equations (Section~\ref{s_pne})
and half-preconditioned normal equations (Section \ref{s_hpne}) are realistic and informative.
In particular, they accurately capture the size of the underlying least squares residual.
\item The accuracy of the solutions depends on the residual
of the original least squares problem -- even though the preconditioned normal equations do not 'know'
about a least squares problem, and even though the half-preconditioned normal equations
do not have an equivalent least squares problem.
\item We justify the choice of perturbations that destroy the symmetry of  the linear system  by illustrating
the shortfall of the ones that do preserve symmetry
 (Appendix~\ref{s_apne}). 
 \item We derive perturbation bounds for right preconditioned Blendenpik-style least squares
 problems (Appendix~\ref{s_apls}).
 However, a comparison with the bounds for the preconditioned
 normal equations is inconclusive. 
 \item Probabilistic condition number bounds demonstrate the effectiveness of randomized preconditioners computed from small amounts of sampling (Section~\ref{s_condno}).
\item With an effective preconditioner, the solutions from the preconditioned normal equations,
when computed with direct methods,
are almost as accurate
as those from the Matlab backslash (mldivide) command, which, for rectangular 
matrices\footnote{https://www.mathworks.com/help/matlab/ref/double.mldivide.html}
is based on a QR decomposition (Section~\ref{s_num}).
 \end{enumerate}

\begin{figure}
\begin{center}
\resizebox{2.9in}{!}
{\includegraphics{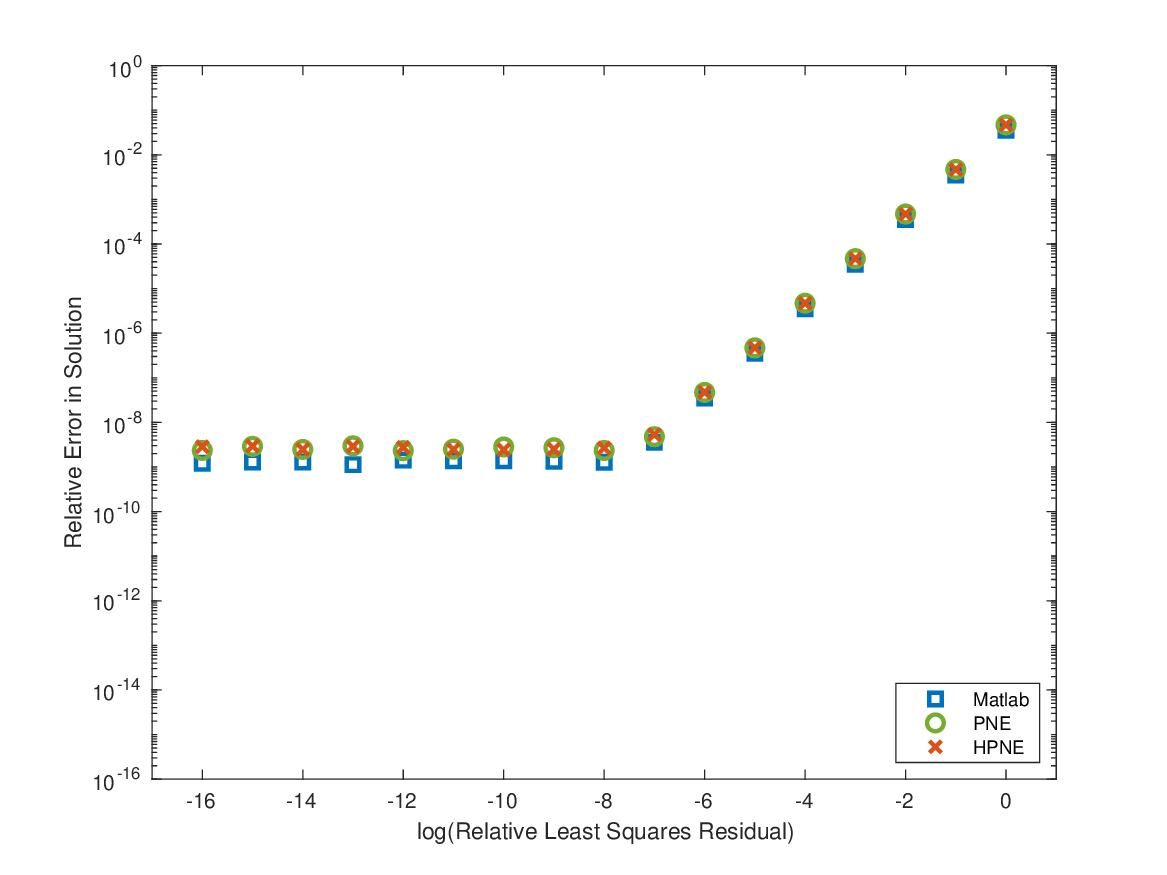}} 
\end{center}
\caption{Relative errors in three different computed solutions $\vhx$ versus logarithm of relative 
least squares residuals $\|\vb-\ma\vx_*\|_2/(\|\ma\|_2\|\vx_*\|_2)$ for 
$\ma\in\real^{6,000\times 1000}$ with condition number $\kappa(\ma)=10^8$,
and preconditioned matrix $\ma_p$ with condition number $\kappa(\ma_p)\approx 4.2$.
The solutions are computed with Matlab backslash (blue squares), preconditioned normal
equations (green circles) and half-preconditioned normal equations (red crosses).}\label{fig1}
\end{figure}

\subsection{Existing Work}
Most existing work on preconditioned normal equations appears to focus
on algorithms rather than conditioning, and in particular
on preconditioners for accelerating the convergence of iterative methods, and
improving their numerical stability via iterative refinement.

A number of papers investigate the solution of nonsingular nonsymmetric systems 
$\ma\vx=\vb$ by solving instead the associated normal equations 
$\ma^T\ma\vx=\ma^T\vb$ via preconditioned iterative methods,
such as CGNE \cite[Section 11.3.9.]{GovL13}, \cite{PT2024}.
Wathen \cite{Wathen2022} gives several examples for the matrix squaring problem: 
if $\mP$ is a good preconditioner for $\ma$, then $\mP^T\mP$ is not
necessarily a good preconditioner for~$\ma^T\ma$. In \cite{Wathen2025} Wathen pursues an 
approach   similar to the half preconditioned normal equations, by  considering nonsymmetric linear systems  $\mb^T\ma=\mb^T\vb$ that are equivalent to the normal
equations, where the columns of $\mb$ represent a
basis for the column space of $\ma$.

Epperly, Greenbaum and Nakatsukasa \cite{Epperly2024,EGN2025} investigate preconditioned 
LSQR combined with iterative refinement.
Lazzarino, Nakatsukasa and Zerbinati
\cite{LNZ2025} consider systems arising from PDE
discretizations, and preconditioned Krylov space methods like
CGNE and LSQR. 
Scott and Tum{\r a}
\cite{ScottTuma2025} consider LSQR preconditioned with incomplete Cholesky factors 
computed
in lower precision.

Carson and Dau{\u z}ickait{\.e} \cite{carson2025} consider the solution of full column-rank least squares problems via iterative
refinement of the semi-normal equations $\mr^T\mr\vx=\ma^T\vb$,  where the residual 
is computed in higher precision than the working accuracy, and observe 
that the semi-normal equations are not sensitive to the size of the least squares residual 
\cite[Section 8]{carson2025}.

\subsection{Notation and Background}
From now on, the Euclidean two-norm is simply denoted by $\|\cdot\|$. For a matrix $\ma\in\rmn$ with 
$\rank(\ma)=n$, the Moore-Penrose inverse is $\ma^{\dagger}=(\ma^T\ma)^{-1}\ma^T$,
the two-norm condition number with respect to left inversion is 
$\kappa(\ma)\equiv\|\ma\|\|\ma^{\dagger}\|$, and the singular values are 
$\sigma_1(\ma)\geq \cdots \geq\sigma_n(\ma)>0$.

To set the context, we give a perturbation bound for the original least squares problem~(\ref{e_ls}).

\begin{lemma}[Fact 5.14 in \cite{IIbook}]\label{l_ls}
Let $\ma, \ma+\me\in\rmn$ with $\rank(\ma)=\rank(\ma+\me)=n$ and 
$\epsilon_A\equiv \|\me\|/\|\ma\|$. Let $\vx_*$ be the solution 
to $\min_{\vx}{\|\ma\vx-\vb\|}$ and $\vhx\neq \vzero$ the solution to $\min_{\vx}{\|(\ma+\me)\vx-\vb\|}$.
Then
\begin{align*}
\frac{\|\vhx-\vx_*\|}{\|\vhx\|}\leq \kappa(\ma)\,\epsilon_A\, 
\left(1+\kappa(\ma)\rho\right), \qquad \text{where}\quad
\rho\equiv \frac{\|\vb-(\ma+\me)\vhx\|}{\|\ma\|\|\vhx\|}.
\end{align*}
\end{lemma}

The condition number of the least squares solution is $\kappa(\ma)\,\max\left\{1,\kappa(\ma)\rho\right\}$,
where $\rho$ represents the relative least squares residual of the 
perturbed problem.
We limit perturbations to those of the matrix $\ma$, and assume an exact right hand side $\vb$, 
because matrix perturbations tend to be much more influential on the sensitivity of least squares problems
than right hand side perturbations.

\section{Perturbation of the preconditioned normal equations}\label{s_pne}
After presenting the assumptions (Section~\ref{s_ass1}), we derive realistic perturbation bounds for the preconditioned normal equations 
(Section~\ref{s_pne1}).

\subsection{Assumptions}\label{s_ass1}
Let $\ma\in\rmn$ have $\rank(\ma)=n$. Let $\mr_s\in\rnn$ be a fixed nonsingular matrix, and
$\ma_p\equiv \ma\mr_s^{-1}$.
The exact preconditioned normal equations are
\begin{equation}\label{e_pne}
\begin{split}
\ma_p^T\ma_p\vy_*&=\ma_p^T\vb\\
\mr_s\vx_*&=\vy_*.
\end{split}
\end{equation}
Since $\ma^T\ma$ and $\mr_s$ are nonsingular, so is the preconditioned matrix $\ma_p^T\ma_p$.

The first step of the preconditioned normal equations (\ref{e_pne}) is mathematically equivalent
to the least squares problem $\min_{\vy}{\|\ma_p\vy-\vb\|_2}$,
which  has the unique solution $\vy_*\equiv \ma_p^{\dagger}\vb$.
Since the right preconditioned
matrix $\ma_p$ has the same column space as $\ma$, the least
squares residual  is equal to that  of the original problem (\ref{e_ls}), 
\begin{align}\label{e_residual}
\vb-\ma_p\vy_*=\vb-\ma\vx_*.
\end{align}

In order to derive realistic perturbation bounds that reflect the numerical errors in Section~\ref{s_num},
we set up the first step with two different perturbations for the 
preconditioned matrix $\ma_p$, so that the resulting linear system is nonsymmetric.
We assume that the triangular system solution in the second step is computed exactly, because
any errors have only a minor, lower order effect.

Let $\me_s\in\rnn$ and $\me_p\in\rmn$,  $\kappa(\mr_s)\epsilon<1$, and
\begin{align}\label{e_perturb}
\ma_1\equiv \ma(\mr_s+\me_s)^{-1},\qquad
\ma_2\equiv \ma_p+\me_p, \qquad 
\epsilon\equiv \max\left\{\tfrac{\|\me_s\|}{\|\mr_s\|}, \tfrac{\|\me_p\|}{\|\ma_p\|}\right\}.
\end{align}
The computed solutions corresponding to (\ref{e_pne}) are modeled as
\begin{eqnarray}
\ma_1^T\ma_2\vhy&=\ma_1^T\vb\label{e_pne1}\\
\mr_s\vhx&=\vhy,\label{e_pne2}
\end{eqnarray}
where $\vhy\neq \vzero$ and $\vhx\neq\vzero$. Remark~\ref{r_2} at the end of the next section
justifies this non-intuitive choice of perturbations.

\subsection{Perturbation bound}\label{s_pne1}
We prove two auxiliary results (Lemmas \ref{l_5} and~\ref{l_4}) for the main statement (Theorem~\ref{t_3}),  
 and give intuition for the
 choice of perturbations (Remark~\ref{r_2}).

The first lemma presents a perturbation bound for the linear system in (\ref{e_pne1}).

\begin{lemma}[Perturbation bound for (\ref{e_pne1})]\label{l_5}
With the assumptions in Section~\ref{s_ass1}, 
\begin{equation*}
\frac{\|\vy_*-\vhy\|}{\|\vhy\|}\leq 
\kappa(\ma_p)\epsilon\,\left(1+\kappa(\ma_p)\, \eta
\left(\frac{\|\vb-\ma_p\vhy\|}{\|\ma_p\|\|\vhy\|}+\epsilon\right)\right),\qquad \text{where}\quad
\eta\equiv \frac{\kappa(\mr_s)}{1-\kappa(\mr_s)\,\epsilon}. 
\end{equation*}
\end{lemma}

\begin{proof}$\ $
Write
\begin{equation*}
\ma_1=\ma(\mr_s+\me_s)^{-1}=\underbrace{\ma\mr_s^{-1}}_{\ma_p}
(\mi+\underbrace{\me_s\mr_s^{-1}}_{\mf})^{-1}=
\ma_p(\mi-\underbrace{(\mi+\mf)^{-1}\mf}_{\mf_p})
=\ma_p(\mi-\mf_p),
\end{equation*}
where 
$\mf\equiv \me_s\mr_s^{-1}$ and $\mf_p\equiv (\mi+\mf)^{-1}\mf$.
Then (\ref{e_pne1}) can be written as 
\begin{equation*}
\ma_1^T(\vb-\ma_p\vhy)=\ma_1^T\me_p\vhy
\end{equation*}
With $\ma_1=\ma_p(\mi-\mf_p)$ this gives
\begin{equation*}
(\mi-\mf_p)^T\ma_p^T(\vb-\ma_p\vhy)=(\mi-\mf_p)^T\ma_p^T\me_p\vhy.
\end{equation*}
Rearrange, 
\begin{align*}
\ma_p^T(\vb-\ma_p\vhy)
&=\mf_p^T\ma_p^T(\vb-\ma_p\vhy)+(\mi-\mf_p)^T\ma_p^T\me_p\vhy\\
&=\mf_p^T\ma_p^T(\vb-\ma_p\vhy)+\ma_p^T\me_p\vhy-\mf_p^T\ma_p^T\me_p\vhy,
\end{align*}
and multiply by $(\ma_p^T\ma_p)^{-1}$, 
\begin{equation*}
\vy_*-\vhy=\ma_p^{\dagger}\me_p\vhy+
(\ma_p^T\ma_p)^{-1}\left(\mf_p^T\ma_p^T(\vb-\ma_p\vhy)
-\mf_p^T\ma_p^T\me_p\vhy\right).
\end{equation*}
Take norms and use the fact that $\kappa(\ma_p^T\ma_p)=\kappa(\ma_p)^2$,
\begin{align*}
\frac{\|\vy_*-\vhy\|}{\|\vhy\|}&
\leq \kappa(\ma_p)\frac{\|\me_p\|}{\|\ma_p\|}+
\kappa(\ma_p)^2\,\|\mf_p\|\,\left(\frac{\|\vb-\ma_p\vhy\|}{\|\ma_p\|\|\vhy\|} +
\frac{\|\me_p\|}{\|\ma_p\|}\right)\\
&\leq \kappa(\ma_p)\,\epsilon+
\kappa(\ma_p)^2\,\|\mf_p\|\,\left(\frac{\|\vb-\ma_p\vhy\|}{\|\ma_p\|\|\vhy\|} +
\epsilon\right).
\end{align*}
At last bound
\begin{equation*}
\|\mf_p\|\leq \frac{\|\mf\|}{1-\|\mf\|}\leq 
\frac{\|\me_s\| \|\mr_s^{-1}\|}{1-\|\me_s\| \|\mr_s^{-1}\|}
\leq \frac{\kappa(\mr_s)\epsilon}{1-\kappa(\mr_s)\epsilon}=\eta\,\epsilon.
\end{equation*}
\end{proof}

The second lemma presents a perturbation bound for the linear system in (\ref{e_pne2}). 
The bound does not exploit the possible triangular structure of $\mr_s$.

\begin{lemma}[Perturbation bound for (\ref{e_pne2})]\label{l_4}
With the assumptions in Section~\ref{s_ass1}, 
\begin{equation*}
\frac{\|\vx_*-\vhx\|}{\|\vhx\|}\leq \kappa(\mr_s)\,\nu\,\frac{\|\vy_*-\vhy\|}{\|\vhy\|}
\qquad \text{where}\qquad \nu\equiv \frac{\|\mr_s\vhx\|}{\|\mr_s\|\|\vhx\|}\leq 1.
\end{equation*}
\end{lemma}

\begin{proof}$\ $
From $\vx_*-\vhx=\mr_s^{-1}(\vy_*-\vhy)$ follows
\begin{equation*}
\frac{\|\vx_*-\vhx\|}{\|\vhx\|}\leq \|\mr_s^{-1}\|\,\frac{\|\vy_*-\vhy\|}{\|\vhx\|}
=\kappa(\mr_s)\, \underbrace{\frac{\|\vhy\|}{\|\mr_s\|\|\vhx\|}}_{\nu}
\,\frac{\|\vy_*-\vhy\|}{\|\vhy\|}.
\end{equation*}
From $\|\vhy\|=\|\mr_s\vhx\|\leq \|\mr_s\|\|\vhx\|$ follows $\nu\leq 1$.
\end{proof}

The main result is a perturbation bound for the preconditioned normal equations (\ref{e_pne}),
based on Lemmas \ref{l_5} and~\ref{l_4}.

\begin{theorem}\label{t_3}
With the assumptions in Section~\ref{s_ass1},
\begin{align*}
\frac{\|\vx_*-\vhx\|}{\|\vhx\|}\leq \kappa(\mr_s)\,\nu\, \kappa(\ma_p)\,\epsilon
\left(1+\kappa(\ma_p)\, \eta
\left(\frac{\|\vb-\ma_p\vhy\|}{\|\ma_p\|\|\vhy\|}+\epsilon\right)\right),
\end{align*}
where 
\begin{align*}
\nu\equiv \frac{\|\mr_s\vhx\|}{\|\mr_s\| \|\vhx\|}\leq 1\qquad \text{and}\qquad
\eta\equiv \frac{\kappa(\mr_s)}{1-\kappa(\mr_s)\,\epsilon}. 
\end{align*}
\end{theorem}

\begin{proof}$\ $
Substitute the bound from Lemma~\ref{l_5} below into Lemma~\ref{l_4}.
\end{proof}

Theorem~\ref{t_3} shows that the solution accuracy of the preconditioned normal equations 
(\ref{e_pne}) depends
on the least squares residual of the original least squares problem (\ref{e_ls}) -- even though the preconditioned normal equations do not 'know' about the least squares problem.
Appendix~\ref{s_ne}
shows that this dependence is also present in the ordinary normal equations.

Theorem~\ref{t_3} implies that, to first order, the  relative error in $\vhx$ is bounded by
\begin{align}\label{e_pne1a}
\frac{\|\vx_*-\vhx\|}{\|\vhx\|}\lesssim \kappa(\ma_p)\,\kappa(\mr_s)\,\epsilon\,\max\left\{
1,\>\kappa(\ma_p)\kappa(\mr_s)
\frac{\|\vb-\ma_p\vhy\|}{\|\ma_p\|\|\vhy\|}\right\}.
\end{align}
That is, if the least squares residual is sufficiently small, so that
\begin{align*}
\kappa(\ma_p)\,\kappa(\mr_s) \frac{\|\vb-\ma_p\vhy\|}{\|\ma_p\|\|\vhy\|}\leq 1,
\end{align*}
then the relative error in $\vhx$ is dominated by $\kappa(\ma_p)\kappa(\mr_s)\,\epsilon$.
Otherwise, the relative error in $\vhx$ is proportional to  the least 
squares residual.

If the preconditioner~$\mr_s$ is effective, so that $\kappa(\ma_p)\leq 10$,
 then the preconditioned normal equations~(\ref{e_pne})
are a numerically stable algorithm for solving least squares problems, because the bound in Theorem~\ref{t_3}
resembles the perturbation bound of the original 
least squares problem in Lemma~\ref{l_ls}.

Why? An effective  preconditioner produces $\kappa(\ma_p)\approx 1$ and
$\kappa(\mr_s)\approx\kappa(\ma)$, so that $\kappa(\ma_p)\kappa(\mr_s)\approx \kappa(\ma)$.
If the relative least squares residuals in (\ref{e_pne1a}) and Lemma~\ref{l_ls}
are about the same, so that
\begin{equation*}
\rho=\frac{\|\vb-\ma_p\vhy\|}{\|\ma_p\|\|\vhy\|}\approx \frac{\|\vb-(\ma+\me)\vhx\|}{\|\ma\|\|\vhx\|},
\end{equation*}
then the condition number in (\ref{e_pne1a}) is about the same as 
that in Lemma~\ref{l_ls},
\begin{align}\label{e_r2a}
\kappa(\ma_p)\kappa(\mr_s)\,\max\left\{1,\,\kappa(\ma_p)\kappa(\mr_s)\rho\right\}
\approx \kappa(\ma)\,\max\left\{1,\kappa(\ma)\rho\right\}.
\end{align}
Furthermore, (\ref{e_residual}) implies that the exact 
least squares residuals of the original and preconditioned problem are the same.
Hence~(\ref{e_r2a}) has the same form as the bound in Lemma~\ref{l_ls}.
The
numerical experiments in Section~\ref{s_num2} illustrate that Theorem~\ref{t_3} is 
informative and realistic.

\begin{remark}\label{r_2}
Why does Theorem~\ref{t_3} have to resort to perturbations that destroy the symmetry of the 
linear system?
It is because perturbations that preserve the symmetry lead to unrealistic condition numbers, 
as illustrated in
Appendix~\ref{s_apne}. Here are the details.

\begin{enumerate}
\item Intuitively we would just apply the perturbation bound for the normal equations in  
Lemma~\ref{l_a2}
to the preconditioned normal equations (\ref{e_pne1}), and then
account for the triangular system solution via Lemma~\ref{l_4}.

This is precisely what Lemma~\ref{l_a3} does: It perturbs all instances of $\ma_p$ by the
same matrix $\me_p$, which leads to the condition number 
\begin{align*}
\kappa(\ma_p)^2\kappa(\mr_s)\max\{1, \rho\},
\end{align*} 
where $\rho $ is a relative least squares residual.
The condition number in Lemma~\ref{l_ls}, and numerical experiments
 indicate that this is too optimistic.

\item Lemma~\ref{l_a6} shows that perturbing all instances of $\mr_s$ 
by the same matrix~$\me_s$ leads to a condition number 
$\kappa(\ma_p)^2\kappa(\mr_s)^2\max\{1, \rho\}$, where $\rho $ is a relative least squares residual.
A comparison with~(\ref{e_r2a}) and numerical experiments  illustrate that this is too pessimistic.
\end{enumerate}
\end{remark}

Perturbation bounds for right-preconditioned least squares solvers,
in the spirit of Blendenpik \cite[Algorithm 1]{Blendenpik} are derived in
Appendix~\ref{s_apls}. However, a comparison with Theorem~\ref{t_3} is inconclusive.

\section{Perturbation of half preconditioned normal equations}\label{s_hpne}
We derive realistic perturbation bounds for the half preconditioned normal equations, under the
following assumptions.

Let $\ma\in\rmn$ have $\rank(\ma)=n$. Let $\mr_s\in\rnn$ be a fixed
nonsingular matrix and $\ma_p\equiv \ma\mr_s^{-1}$.
The exact half preconditioned normal equations are 
\begin{align}\label{e_hp}
\ma_p^T\ma\vx_*=\ma_p^T\vb.
\end{align}
Since $\ma^T\ma$ and $\mr_s$ are nonsingular, so is the half preconditioned
matrix $\ma_p^T\ma$.

The subsequent Remark~\ref{r_4} justifies the non-intuitive choice of perturbations
in Theorem~\ref{t_2}.

\begin{theorem}\label{t_2}Let  $\me_s\in\rnn$, $\me_A\in\rmn$, 
\begin{align*}
\ma_1\equiv \ma(\mr_s+\me_s)^{-1},\qquad \ma_2\equiv \ma+\me_A, \qquad
\epsilon\equiv \max\left\{\frac{\|\me_s\|}{\|\mr_s\|}, \frac{\|\me_A\|}{\|\ma\|}\right\},
\end{align*}
$\kappa(\mr_s)\epsilon<1$, and
\begin{align}
\ma_1^T\ma_2\vhx=\ma_1^T\vb.\label{e_t21}
\end{align}
If  $\vhx\neq \vzero$, then
\begin{align*}
\frac{\|\vx_*-\vhx\|}{\|\vhx\|}\leq \kappa(\ma_p^T\ma)\, \nu\, \epsilon\,
\left(1+\eta\frac{\|\vb-\ma\vhx\|}{\|\ma\|\|\vhx\|}+\eta\epsilon\right),
\end{align*}
where 
\begin{align*}
\nu\equiv \frac{\|\ma_p\|\|\ma\|}{\|\ma_p^T\ma\|}\geq 1,\qquad 
\eta\equiv \frac{\kappa(\mr_s)}{1-\kappa(\mr_s)\epsilon}.
\end{align*}
\end{theorem}

\begin{proof}$\ $
Write
\begin{equation*}
\ma_1=\ma(\mr_s+\me_s)^{-1}=\underbrace{\ma\mr_s^{-1}}_{\ma_p}
(\mi+\underbrace{\me_s\mr_s^{-1}}_{\mf})^{-1}=
\ma_p(\mi-\underbrace{(\mi+\mf)^{-1}\mf}_{\mf_p})
=\ma_p(\mi-\mf_p),
\end{equation*}
where 
$\mf\equiv \me_s\mr_s^{-1}$ and $\mf_p\equiv (\mi+\mf)^{-1}\mf$.
Then (\ref{e_t21}) can be written as 
\begin{equation*}
\ma_1^T(\vb-\ma\vhx)=\ma_1^T\me_A\vhx.
\end{equation*}
With $\ma_1=\ma_p(\mi-\mf_p)$, this gives
\begin{equation*}
(\mi-\mf_p)^T\ma_p^T(\vb-\ma\vhx)=(\mi-\mf_p)^T\ma_p^T\me_A\vhx.
\end{equation*}
Rearrange
\begin{equation*}
\ma_p^T(\vb-\ma\vhx)=\mf_p^T\ma_p^T(\vb-\ma\vhx)+(\mi-\mf_p)^T\ma_p^T\me_A\vhx
\end{equation*}
and multiply by $(\ma_p^T\ma)^{-1}$,
\begin{equation*}
\vx_*-\vhx=(\ma_p^T\ma)^{-1}\left(\mf_p^T\ma_p^T(\vb-\ma\vhx)
+(\mi-\mf_p)^T\ma_p^T\me_A\vhx\right).
\end{equation*}
Take norms
\begin{equation*}
\frac{\|\vx_*-\vhx\|}{\|\vhx\|}\leq \kappa(\ma_p^T\ma)\,
\underbrace{\frac{\|\ma_p\|\|\ma\|}{\|\ma_p^T\ma\|}}_{\nu}\,
\left(\|\mf_p\|\,\frac{\|\vb-\ma\vhx\|}{\|\ma\|\|\vhx\|} +
(1+\|\mf_p\|)\underbrace{\frac{\|\me_A\|}{\|\ma\|}}_{\leq \epsilon}\right),
\end{equation*}
and bound
\begin{equation*}
\|\mf_p\|\leq \frac{\|\mf\|}{1-\|\mf\|}\leq 
\frac{\|\me_s\| \|\mr_s^{-1}\|}{1-\|\me_s\| \|\mr_s^{-1}\|}
\leq \frac{\kappa(\mr_s)\epsilon}{1-\kappa(\mr_s)\epsilon}=\eta\epsilon.
\end{equation*}
\end{proof}

Theorem~\ref{t_2} implies that the solution accuracy of the half preconditioned normal equations
(\ref{e_hp}) depends on the least squares residual of the original least squares problem (\ref{e_ls})
-- even though (\ref{e_hp}) are not mathematically equivalent to a least squares problem.

Theorem~\ref{t_2} implies that, to first order, the relative error in $\vhx$ is bounded by
\begin{equation}\label{e_pne2a}
\frac{\|\vx_*-\vhx\|}{\|\vhx\|}\lesssim  \kappa(\ma_p^T\ma)\, \nu\, \epsilon
\left(1+\kappa(\mr_s)\frac{\|\vb-\ma\vhx\|}{\|\ma\|\|\vhx\|}\right).
\end{equation}
If the preconditioner~$\mr_s$ is effective, so that $\kappa(\ma_p)\leq 10$, then the half preconditioned normal 
equations~(\ref{e_hp}) are a numerically stable algorithm for solving the least squares
problem, because the bounds in Theorem~\ref{t_2}
and~(\ref{e_pne2a}) resemble the perturbation bound of the original 
least squares problem in Lemma~\ref{l_ls}.

Why?  The singular values of $\ma_p^T\ma$ are bounded by
\begin{align*}
\sigma_n(\ma_p)\sigma_j(\ma)\leq \sigma_j(\ma_p^T\ma)\leq \sigma_1(\ma_p)\sigma_j(\ma),
\qquad 1\leq j\leq n.
\end{align*}
Thus $\kappa_2(\ma_p^T\ma)\leq \kappa_2(\ma_p)\kappa_2(\ma)$.
An effective preconditioner produces $\kappa(\ma_p)\approx 1$ and
$\kappa(\mr_s)\approx\kappa(\ma)$, so that $\kappa(\ma_p^T\ma)\approx \kappa(\ma)$.
Furthermore
$\|\ma_p\|_2\approx 1$ and $\|\ma_p^T\ma\|_2\approx \|\ma\|_2$
implies $\nu\approx 1$. 
If the relative least squares residuals in (\ref{e_pne2a}) and Lemma~\ref{l_ls}
are about the same, so that
\begin{equation*}
\rho=\frac{\|\vb-\ma\vhx\|}{\|\ma\|\|\vhx\|}\approx \frac{\|\vb-(\ma+\me)\vhx\|}{\|\ma\|\|\vhx\|},
\end{equation*}
then the condition number in (\ref{e_pne2a}) is about the same as 
that in Lemma~\ref{l_ls},
\begin{align}\label{e_r3a}
\kappa(\ma_p^T\ma)\nu\max\left\{1,\, \kappa(\mr_s)\rho\right\}\approx 
\kappa(\ma)\,\max\left\{1,\kappa(\ma)\rho\right\}.
\end{align}
The
numerical experiments in Section~\ref{s_num3} illustrate that Theorem~\ref{t_2} is 
informative and realistic, with $\nu\leq 2$ for the randomized preconditioner.

In the special case
 $\ma_p=\ma$ for the normal equations (\ref{e_ne1}), we have $\nu=1$
 and Theorem~\ref{t_2} essentially reduces to Lemma~\ref{l_a2}.

\begin{remark}\label{r_4}
Why does Theorem~\ref{t_2}  have to resort to different perturbations for $\ma$?
It is because the same perturbations lead to unrealistic condition numbers.

Lemma~\ref{l_a7} shows that perturbing $\ma_p$ and $\ma$
leads to a condition number 
$\kappa(\ma_p^T\ma)\nu\max\{1, \rho\}$, where $\rho $ is a relative least squares residual.
Lemma~\ref{l_ls} and the numerical experiments indicate that this is too optimistic.
\end{remark}

\section{Probabilistic Condition Number Bounds}\label{s_condno}
We review the randomized sampling approach for the randomized preconditioner (Section~\ref{s_sample}) and derive condition number bounds for the preconditioner
and the preconditioned matrices (Section~\ref{s_precond}).

\subsection{Randomized Preconditioner}\label{s_sample}
The randomized preconditioner, motivated by the least squares solver 
\textit{Blendenpik} \cite{Blendenpik},
is computed with the pseudocode in Algorithm~\ref{alg_precond}.

 Given $\ma\in\rmn$ with $\rank(\ma)=n$, we produce a smaller dimensional matrix 
 by sampling $c$ rows from the smoothed matrix 
$\ \ma_s\equiv\ms\,\mathcal{F}\,\ma$ uniformly and with replacement. 
The matrix $\mathcal{F}=\mf\md\in\rmm$ is a random orthogonal matrix,  where
$\mf$ is a discrete cosine transform (DCT-2),
\begin{equation*}
\mf_{ij}= \sqrt{\frac{2}{m}}\>\cos\left(\frac{\pi}{2m}(2j-1)(i-1)\right) \qquad
1\leq i,j\leq m
\end{equation*}
and $\md$ is random diagonal matrix whose diagonal elements are
$\md_{jj}=\pm 1$ with probability 1/2, $1\leq j\leq m$. 
The matrix $\ms\in\real^{c\times m}$ samples $c$ rows $k_1,\ldots k_c$
from the identity $\mi_m$, uniformly and with replacement,
\begin{equation*}
\mi_m=\begin{bmatrix} 1& & \\ &\ddots &\\ & &1\end{bmatrix}= 
\begin{bmatrix} \ve_1^T \\ \vdots \\ \ve_m^T\end{bmatrix}\in\rmm \qquad\quad
\ms\equiv \sqrt{\tfrac{m}{c}}\begin{bmatrix} \ve_{k_1}^T \\ \vdots \\ \ve_{k_c}^T\end{bmatrix}
\in\real^{c\times m}.
\end{equation*}
In expectation we have $\E[\ms^T\ms]=\mi_m$.
The purpose of the row mixing matrix $\mathcal{F}$
is to improve the coherence, so that the subsequent uniform sampling via $\ms$ is effective.

\begin{algorithm}
\caption{Computation of the randomized preconditioner}\label{alg_precond}
\begin{algorithmic}
\Require Given $\ma\in\rmn$ with $\rank(\ma)=n$, sampling amount $c\geq n$
\State Sample $c$ rows from smoothed matrix: $\ma_s\equiv\ms\mathcal{F}\ma$
\State Compute preconditioner $\mr_s\in\rnn$ from thin QR decomposition $\ma_s=\mq_s\,\mr_s$
\State Precondition the matrix: $\ma_p\equiv \ma\,\mr_s^{-1}$
\end{algorithmic}
\end{algorithm}
\subsection{Condition Number Bounds}\label{s_precond}
After presenting the assumptions (Assumptions~\ref{ass2}) and an auxiliary deterministic result (Lemma~\ref{l_prob}), we present probabilistic  bounds for the singular values
of the preconditioned matrix (Theorem~\ref{t_prob}), followed by probabilistic bounds on the 
condition numbers of the preconditioner and the preconditioned matrices (Theorem~\ref{t_prob2}).

\begin{assumptions}\label{ass2}
 Let $\ma\in\rmn$ with $\rank(\ma)=n$ and thin QR factorization $\ma=\mq\mr$ where $\mq\in\rmn$
with $\mq^T\mq=\mi_n$. Let $\ms\in\real^{c\times n}$  sample $c$ rows uniformly, independently, and with replacement.
Let $\mathcal{F}\in\rmm$ be a random orthogonal matrix,  and let  $\mathcal{F}\mq$ 
have coherence $\mu\equiv \max_{1\leq i\leq m}{\|\ve_i^T\mathcal{F}\mq\|_2^2}$.
Let the sampled matrix be $\ma_s\equiv \ms\mathcal{F}\mq$.
\end{assumptions}

 We express the singular values of the preconditioned matrix $\ma_p$  
 in terms of the singular values of~$\ms\mathcal{F}\mq$.

 \begin{lemma}[Lemma 4.1 in \cite{GI24}]\label{l_prob} 
 Under Assumptions~\ref{ass2}, if also  $\rank(\ma_s)=n$ in Algorithm~\ref{alg_precond}, then
 \begin{equation*}
\sigma_i(\ms\mathcal{F}\mq)=1/\sigma_{n-i+1}(\ma_p), \qquad 1\leq i\leq n,
\end{equation*}
and $\kappa(\ms\mathcal{F}\mq)=\kappa(\ma_p)$.
\end{lemma}

We extend  \cite[Corollary 4.2]{IpsW12} by deriving lower and bounds for the singular values
of the preconditioned matrix. The bounds below hold
for all singular value simultaneously.

\begin{theorem}\label{t_prob}
 Under Assumptions~\ref{ass2}, for any $0<\epsilon<1$ and $0<\delta<1$, if
\begin{equation*}
c\geq 2m\,\mu\,\left(1+\frac{\epsilon}{3}\right)\,\frac{\ln{(n/\delta)}}{\epsilon^2}
\end{equation*} 
then with probability at least $1-\delta$ 
\begin{equation*}
\sqrt{\frac{1}{1+\epsilon}}\leq \sigma_j(\ma_p)\leq \sqrt{\frac{1}{1-\epsilon}},
\qquad 1\leq j\leq n.
\end{equation*}
\end{theorem}

\begin{proof}$\ $
Set $\mx\equiv (\ms\mathcal{F}\mq)^T(\ms\mathcal{F}\mq)\in\rnn$ where
$\E[\mx]=\mi_n=(\mathcal{F}\mq)^T(\mathcal{F}\mq)$.
Apply steps 1-4 in the proof of \cite[Theorem 7.5]{HoI15} to deduce
\begin{equation*}
\myP[\|\mx-\mi_n\|_2>\epsilon]\leq n\exp\left(\frac{-c\epsilon^2}{2m\mu(1+\epsilon/3)}
\right).
\end{equation*}
Then solve for $c$. Weyl's theorem \cite[Corollary 8.1.6]{GovL13}
implies for the eigenvalues
\begin{equation*}
\max_{1\leq j\leq n}{|\lambda_j(\mx)-1|}\leq \|\mx-\mi_n\|_2\leq \epsilon.
\end{equation*}
Hence $1-\epsilon\leq \lambda_j(\mx)\leq 1+\epsilon$, $1\leq j\leq n$.
The result follows from $\lambda_j(\mx)=\sigma_j(\ms\mathcal{F}\mq)^2$ 
and Lemma~\ref{l_prob}.
\end{proof}

Theorem~\ref{t_prob} implies probabilistic lower and upper bounds for the condition number of
the preconditioned matrix \cite{Blendenpik,IpsW12,RT08},
 \begin{equation}\label{e_condAp}
\sqrt{\frac{1-\epsilon}{1+\epsilon}}\leq 
\kappa(\ma_p)\leq \sqrt{\frac{1+\epsilon}{1-\epsilon}}.
\end{equation}
Although $\kappa(\ma_p)\geq 1$, the above lower bound illustrates similarities
with the subsequent Theorem~\ref{t_prob2}.

  The purpose of the row mixing matrix $\mathcal{F}$ is to improve the coherence so that 
  $\mu\approx n/m$, thereby making the lower bound for $c$ independent of $m$ \cite[Section 3.2]{Blendenpik}.

We apply Theorem~\ref{t_prob} to derive probabilistic lower and upper bounds on the condition 
numbers of: the preconditioner, the matrix in the preconditioned normal  equations, and the matrix in the half-preconditioned normal equations.

\begin{theorem}\label{t_prob2}
Under Assumptions~\ref{ass2}, for any $0<\epsilon<1$ and $0<\delta<1$, if
\begin{equation*}
c\geq 2m\,\mu\,\left(1+\frac{\epsilon}{3}\right)\,\frac{\ln{(n/\delta)}}{\epsilon^2}
\end{equation*} 
then with probability at least $1-\delta$, the following hold simultaneously,
\begin{equation*}
\sqrt{\frac{1-\epsilon}{1+\epsilon}}
\,\kappa(\ma)\leq\kappa(\mr_s)\leq \sqrt{\frac{1+\epsilon}{1-\epsilon}}\,\kappa(\ma), \qquad
\frac{1-\epsilon}{1+\epsilon}\leq \kappa(\ma_p^T\ma_p)\leq \frac{1+\epsilon}{1-\epsilon},\qquad
\end{equation*}
and
\begin{equation*}
\sqrt{\frac{1-\epsilon}{1+\epsilon}}\,\kappa(\ma)\leq \kappa(\ma_p^T\ma)
\leq \sqrt{\frac{1+\epsilon}{1-\epsilon}}\,\kappa(\ma).
\end{equation*}
\end{theorem}

\begin{proof}$\ $
The bound for $\kappa(\mr_s)$ follows from the application of the singular value product inequalities  \cite[(7.3.14)]{HoJoI} to $\ma=\ma_p\mr_s$,
\begin{equation*}
\sigma_n(\ma_p)\sigma_j(\mr_s)\leq \sigma_j(\ma)\leq \sigma_1(\ma_p)\sigma_j(\mr_s),\qquad
1\leq j\leq n,
\end{equation*}
and (\ref{e_condAp}). The remaining inequalities are derived analogously. 
\end{proof}

Theorem~\ref{t_prob2} implies that for small $\delta$ and $\epsilon$, the condition numbers
of $\mr_s$ and $\ma_p^T\ma$ are close to that of $\kappa(\ma)$; and the condition number
of $\ma_p^T\ma_p$ is close to one.

\section{Numerical Experiments}\label{s_num}
We illustrate the accuracy of the preconditioned and 
half-preconditioned normal equations, and the perturbation bounds. After
the set up of the numerical experiments (Section~\ref{s_num1}), we present numerical
experiments for the preconditioned normal equations (Section~\ref{s_num2}),
 the half preconditioned normal equations (Section~\ref{s_num3}), and for both when
 the matrices are highly illconditioned (Section~\ref{s_num4}).

\subsection{Set up of Experiments}\label{s_num1}
Algorithm~\ref{alg_exact} presents Matlab pseudocode for
the  computation of the `exact' quantities in the least squares problem (\ref{e_ls}),
as motivated by \cite[Section 1.5]{MNTW24}.

\begin{algorithm}[!t]
\caption{Constructing the least squares problem}\label{alg_exact}
\begin{algorithmic}
\Require Matrix dimensions $m$ and $n$, and condition number $\kappa$
\State $\qquad\ $  Least squares residual norm $\eta_r$
\Ensure Matrix $\ma\in\rmn$ with $\kappa(\ma)=\kappa$,
righthand side $\vb\in\real^m$
\State $\qquad\quad$ Solution $\vx_*\in\rn$ with $\|\vx_*\|=1$
\State $\qquad\quad$ Least squares residual $\ve\equiv \vb-\ma\vx_*\in\real^m$
with $\|\ve\|=\eta_r$
\medskip

\State \Comment{Compute $\ma$}
\State Compute orthogonal matrix $\mq=\begin{bmatrix}\mq_1&\mq_2\end{bmatrix}\in\rmm$ with $\mq_1\in\rmn$
\State Compute upper triangular matrix $\mr\in\rnn$ with $\kappa(\mr)=\kappa$
\State Multiply $\ma=\mq_1\mr\qquad$
\Comment{Thin QR with $\range(\mq_1)=\range(\ma)$}
\medskip

\State \Comment{Compute solution $\vx_*$ with $\|\vx_*\|=1$}
\State $\vx = \texttt{randn}(n, 1)\qquad$ \Comment{Standard random normal vector}
\State $\vx_*=\vx/\|\vx\|$
\medskip

\State \Comment{Compute least squares residual}
\State $\ve_r= \mq_2\mq_2^T\ \texttt{randn}(m,1)\qquad$
\Comment{noisevector $\ve_r$ orthogonal to $\range(\ma)$}
\State $\ve = \eta_r\,\ve_{r}/\|\ve_{r}\|\qquad$
 \Comment{Absolute residual  norm $\|\ma\vx_*-\vb\|=\eta_r$}
 \medskip
 
 \Comment{Compute righthand side $\vb$}
\State $\vb = \ma\vx_*+\ve\qquad $
\end{algorithmic}
\end{algorithm}

We choose matrices $\ma$ with two norm $\|\ma\|=1$,
$m=6,000$ rows, and a number of columns equal to $n=400$ and $n=1,000$.
The condition numbers are $\kappa(\ma)=10^8$,
at which point the ordinary normal equations (\ref{e_ne1})
are too ill-conditioned.
The sampling amount for the preconditioner in Algorithm~\ref{alg_precond} is  $c=3n$.

Since $\|\ma\|=\|\vx_*\|=1$,
the absolute least squares residuals $\|\ma\vx_*-\vb\|$
are equal to the relative least squares residuals $\frac{\|\ma\vx_*-\vb\|}{\|\ma\|\|\vx_*\|}$,
and they vary in norm from $10^{-16}$ all the way up to~1.

We use the IEEE double precision machine epsilon
$\mathtt{eps}\equiv2^{-52}\approx 2.22\cdot 10^{-16}$ in the perturbation bounds.

For the least squares solution, we compute a bound that is a slight variation of that in 
Lemma~\ref{l_ls},
\begin{equation}\label{e_ls1}
\frac{\|\vhx-\vx_*\|}{\|\vhx\|}\lesssim\kappa(\ma)\,\mathtt{eps}\, 
\left(1+\kappa(\ma)\frac{\|\vb-\ma\vhx\|}{\|\ma\|\|\vhx\|}\right).
\end{equation}
With $\kappa(\ma)=10^8$,
the least squares residual starts to dominate the bound once it increases beyond $10^{-8}$.

\subsection{Preconditioned Normal Equations}\label{s_num2}
We illustrate the accuracy of the preconditioned normal equations and their bounds
(Figures \ref{fig2} and~\ref{fig3}).
The matrix in Figure~\ref{fig2} has 400 columns, while the one in Figure~\ref{fig3} has 1,000
columns.

We compute the perturbation bound in Theorem~\ref{t_3} as 
\begin{equation}\label{e_pne4}
\frac{\|\vx_*-\vhx\|}{\|\vhx\|}\leq \kappa(\mr_s)\,\kappa(\ma_p)\nu\,\mathtt{eps}
\left(1+\kappa(\ma_p)\, \eta
\left(\frac{\|\vb-\ma_p\vhy\|}{\|\ma_p\|\|\vhy\|}+\mathtt{eps}\right)\right),
\end{equation}
where 
\begin{equation*}
\nu\equiv \frac{\|\mr_s\vhx\|}{\|\mr_s\| \|\vhx\|}\leq 1\qquad \text{and}\qquad
\eta\equiv \frac{\kappa(\mr_s)}{1-\kappa(\mr_s)\,\mathtt{eps}}.
\end{equation*}

Figures \ref{fig2} and~\ref{fig3} illustrate that
the computed solutions of the preconditioned normal 
equations~(\ref{e_pne}) are almost as accurate as the Matlab solutions. Compared
with the actual error, the bound~(\ref{e_pne4}) is of the same quality as the 
traditional bound~(\ref{e_ls1}). In particular, (\ref{e_pne4}) captures the increase in the least squares residual.

\subsection{Half Preconditioned Normal Equations}\label{s_num3}
We illustrate the accuracy of the half preconditioned normal equations and their bounds
(Figures \ref{fig5} and~\ref{fig6}).
The matrix in Figure~\ref{fig5} has 400 columns, while the one in Figure~\ref{fig6} has 1,000
columns.

We compute the bound from Theorem~\ref{t_2} as
\begin{equation}\label{e_hpne4}
\frac{\|\vx_*-\vhx\|}{\|\vhx\|}\leq \kappa(\ma_p^T\ma)\, \nu\, \mathtt{eps}\,
\left(1+\eta\frac{\|\vb-\ma\vhx\|}{\|\ma\|\|\vhx\|}+\eta\mathtt{eps}\right),
\end{equation}
where 
\begin{equation*}
\nu\equiv \frac{\|\ma_p\|\|\ma\|}{\|\ma_p^T\ma\|}\geq 1,\qquad 
\eta\equiv \frac{\kappa(\mr_s)}{1-\kappa(\mr_s)\mathtt{eps}}.
\end{equation*}

Figures \ref{fig5} and~\ref{fig6} illustrate that
the computed solutions of the half preconditioned normal 
equations~(\ref{e_hp}) are almost as accurate as the Matlab solutions. Compared
with the actual error, the bound~(\ref{e_hpne4}) is of the same quality as the 
traditional bound~(\ref{e_ls1}). In particular, (\ref{e_hpne4}) captures the increase in the least squares residual.

\begin{figure}
\begin{center}
\resizebox{3in}{!}
{\includegraphics{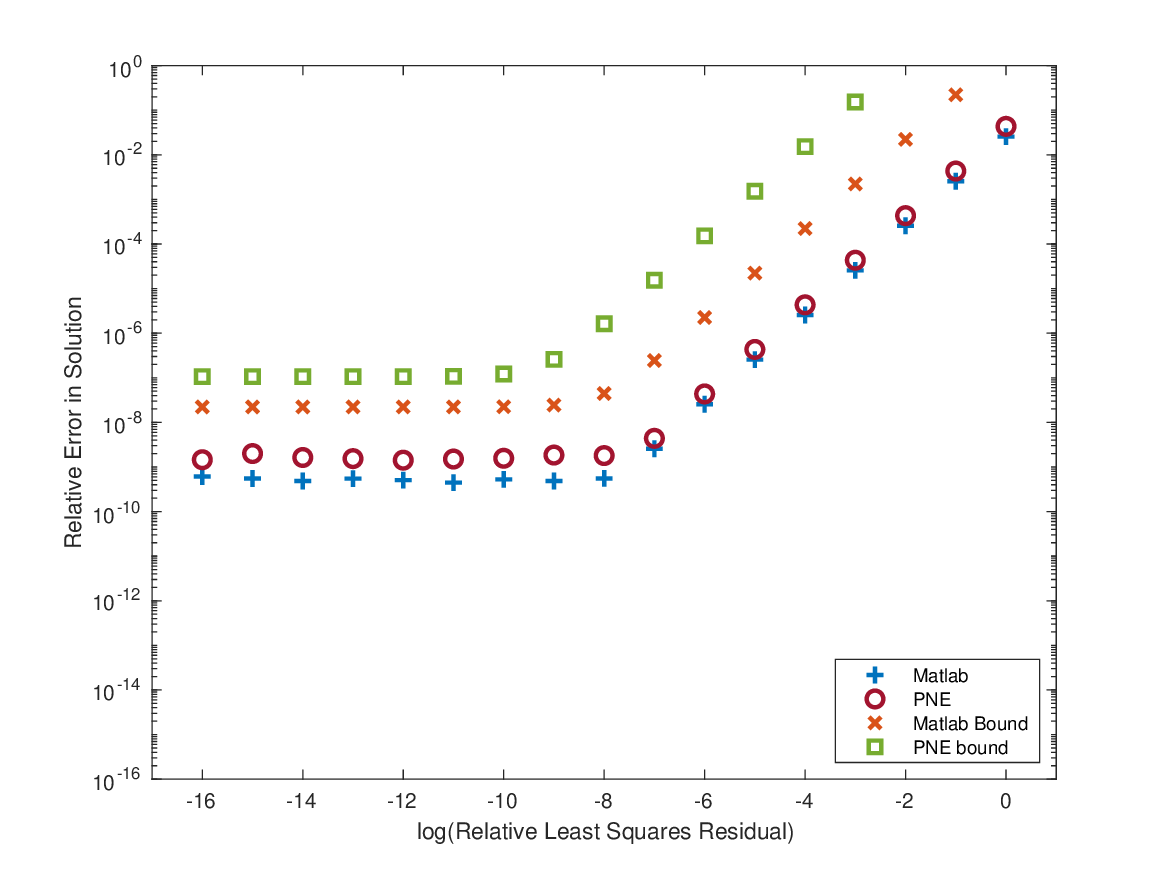}} 
\end{center}
\caption{Preconditioned normal equations: Relative errors in the computed solutions $\vhx$ 
and perturbation bounds versus logarithm of relative 
least squares residuals $\|\vb-\ma\vx_*\|/(\|\ma\|\|\vx_*\|)$ for 
$\ma\in\real^{6,000\times 400}$ with condition number $\kappa(\ma)=10^8$,
and preconditioned matrix $\ma_p$ with condition number $\kappa(\ma_p)\approx 3.82$.
Shown are the Matlab backslash solutions (blue plusses) and the bound~(\ref{e_ls1}) (red crosses); and the solutions from the preconditioned normal
equations (magenta circles) and the bound (\ref{e_pne4}) (green squares).}
\label{fig2}
\end{figure}

\begin{figure}
\begin{center}
\resizebox{2.9in}{!}
{\includegraphics{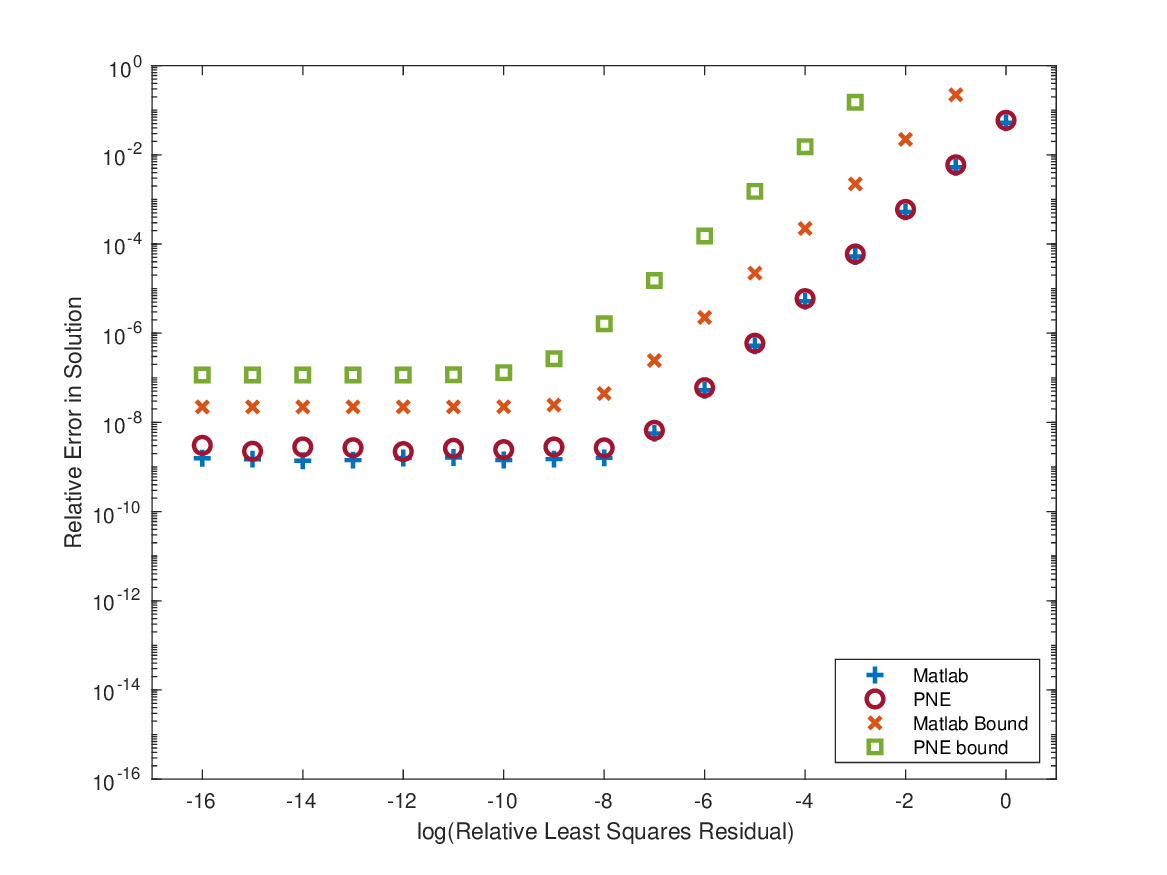}} 
\end{center}
\caption{Preconditioned normal equations: Relative errors in the computed solutions $\vhx$ 
and perturbation bounds versus logarithm of relative 
least squares residuals $\|\vb-\ma\vx_*\|/(\|\ma\|\|\vx_*\|)$ for 
$\ma\in\real^{6,000\times 1000}$ with condition number $\kappa(\ma)=10^8$,
and preconditioned matrix $\ma_p$ with condition number $\kappa(\ma_p)\approx 4.2$.
Shown are the Matlab backslash solutions (blue plusses) and the bound~(\ref{e_ls1}) (red crosses); and the solutions from the preconditioned normal
equations (magenta circles) and the bound (\ref{e_pne4}) (green squares).}\label{fig3}
\end{figure}

\begin{figure}
\begin{center}
\resizebox{2.9in}{!}
{\includegraphics{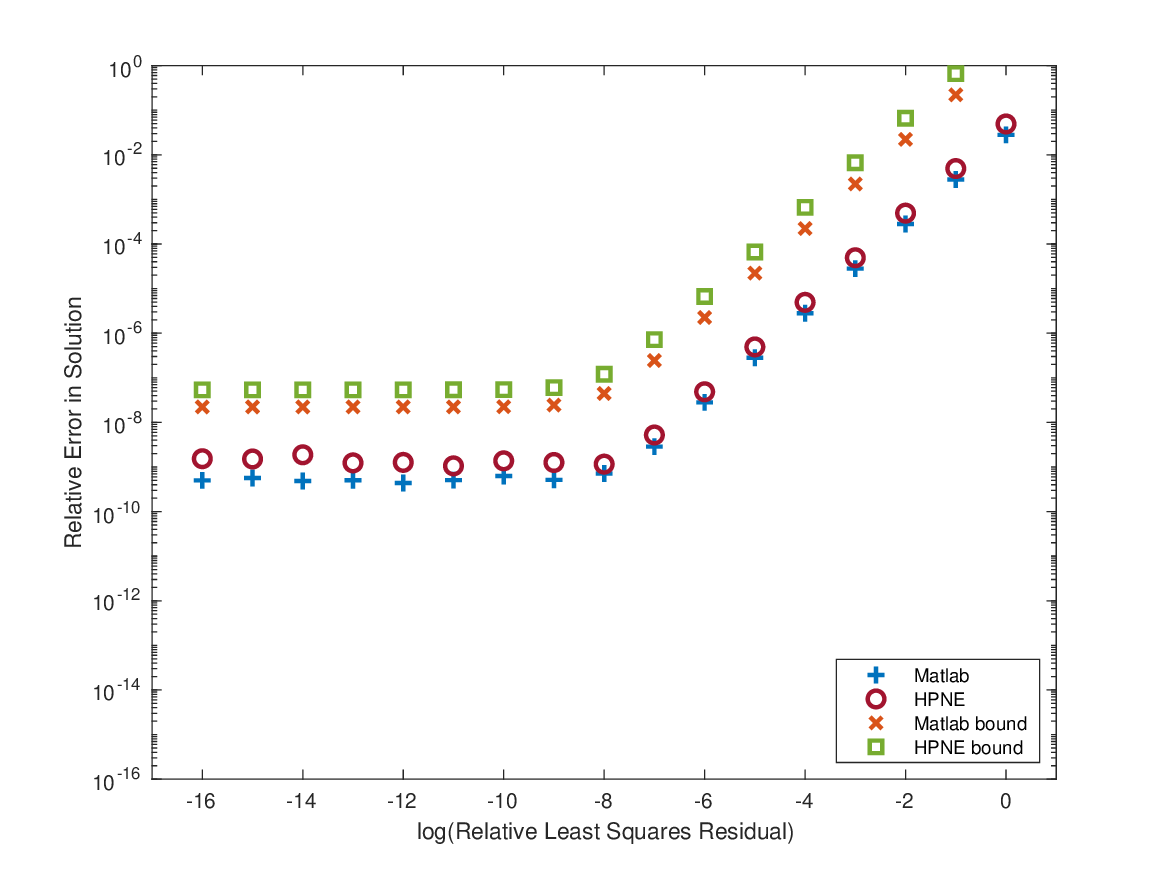}} 
\end{center}
\caption{Half preconditioned normal equations: Relative errors in the computed solutions $\vhx$ 
and perturbation bounds versus logarithm of relative 
least squares residuals $\|\vb-\ma\vx_*\|/(\|\ma\|\|\vx_*\|)$ for 
$\ma\in\real^{6,000\times 1,000}$ with condition number $\kappa(\ma)=10^8$,
and preconditioned matrix $\ma_p$ with condition number $\kappa(\ma_p)\approx 4.28$.
Shown are the Matlab backslash solutions (blue plusses) and the bound~(\ref{e_ls1}) (red crosses); and the solutions from the half-preconditioned normal
equations (magenta circles) and the bound (\ref{e_hpne4}) (green squares).}\label{fig5}
\end{figure}

\begin{figure}
\begin{center}
\resizebox{2.9in}{!}
{\includegraphics{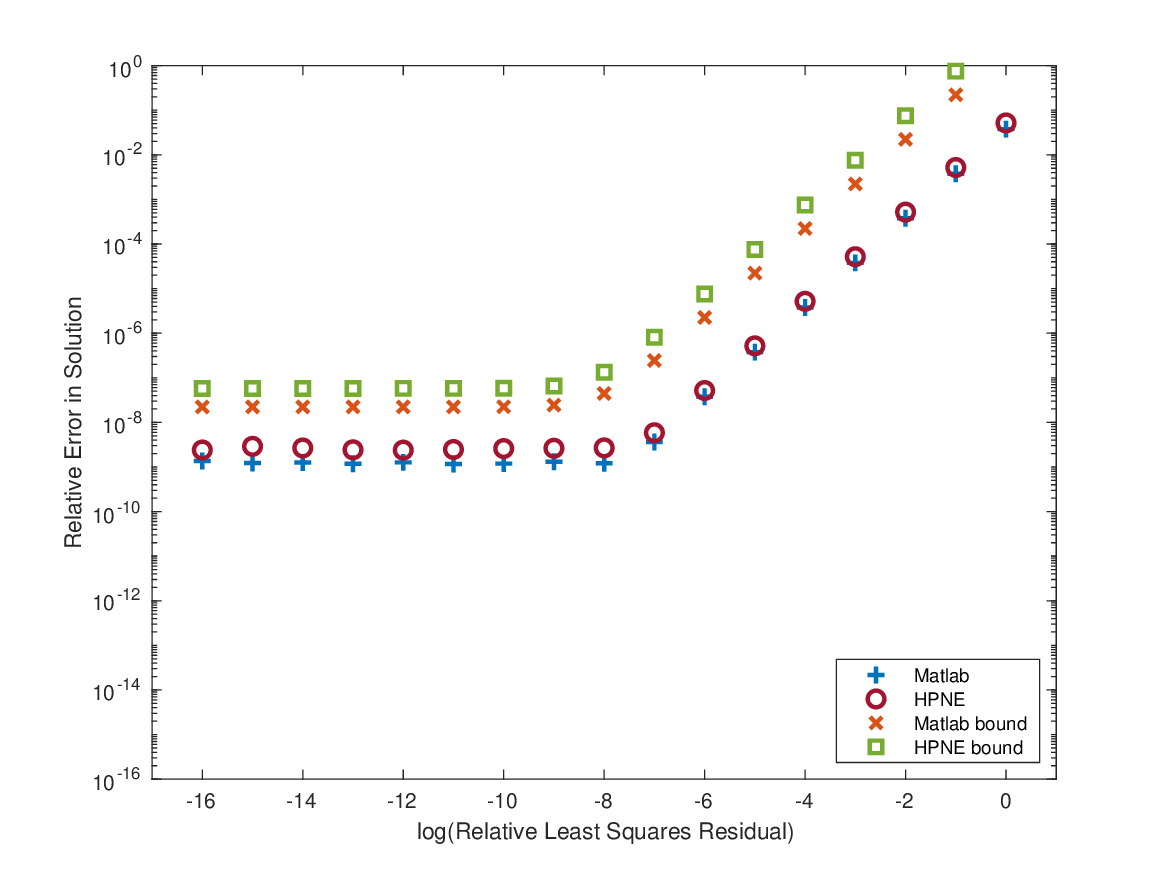}} 
\end{center}
\caption{Half preconditioned normal equations: Relative errors in the computed solutions $\vhx$ 
and perturbation bounds versus logarithm of relative 
least squares residuals $\|\vb-\ma\vx_*\|/(\|\ma\|\|\vx_*\|)$ for 
$\ma\in\real^{6,000\times 400}$ with condition number $\kappa(\ma)=10^8$,
and preconditioned matrix $\ma_p$ with condition number $\kappa(\ma_p)\approx 3.8$.
Shown are the Matlab backslash solutions (blue plusses) and the bound~(\ref{e_ls1}) (red crosses); and the solutions from the half-preconditioned normal
equations (magenta circles) and the bound (\ref{e_hpne4}) (green squares).}\label{fig6}
\end{figure}

\subsection{Highly illconditioned matrices}\label{s_num4}
We illustrate the accuracy of the preconditioned and half preconditioned normal equations
even for highly illconditioned matrices (Figure~\ref{fig7}).

With $\kappa(\ma)=10^{12}$, 
the least squares residual starts to dominate the bound once it increases beyond $10^{-12}$.
Figure~\ref{fig7} illustrates that preconditioned and half preconditioned normal equations
maintain an accuracy of $\kappa(\ma)\mathtt{eps}\approx 10^{-4}$ until the least squares residual
increases beyond $10^{-12}$.

\begin{figure}
\begin{center}
\resizebox{2.9in}{!}
{\includegraphics{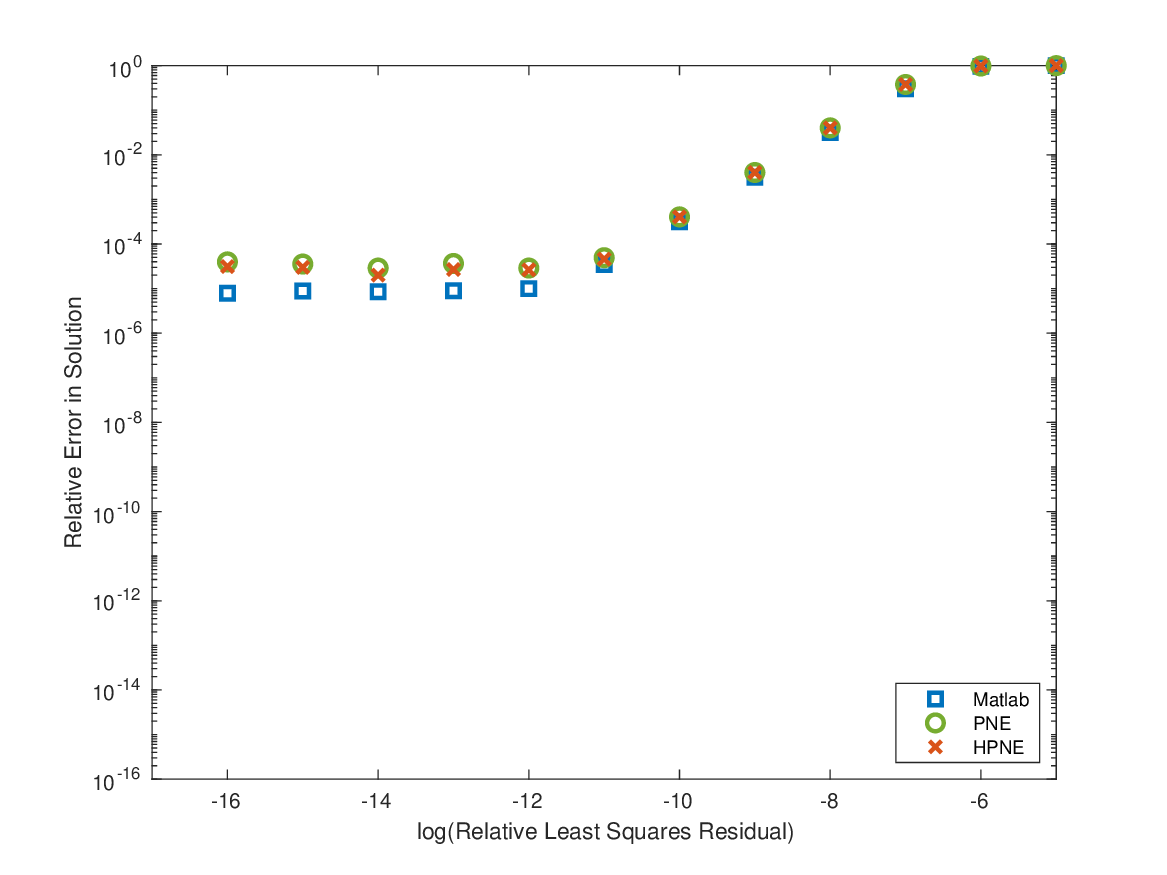}} 
\end{center}
\caption{Relative errors in three different computed solutions $\vhx$ versus logarithm of relative 
least squares residuals $\|\vb-\ma\vx_*\|/(\|\ma\|\|\vx_*\|)$ for 
$\ma\in\real^{6,000\times 1000}$ with condition number $\kappa(\ma)=10^{12}$,
and preconditioned matrix $\ma_p$ with condition number $\kappa(\ma_p)\approx 4.3$.
The solutions are computed with Matlab (blue squares), preconditioned normal
equations (green circles) and half preconditioned normal equations (red crosses).}\label{fig7}
\end{figure}

\section{Future Work}
Our perturbation analysis and numerical experiments show that the normal equations, when preconditioned either on both sides or else only on the left side by a randomized preconditioner,
are well conditioned. When solved with direct methods, they
produce a solution that 
is almost as accurate as the one from the QR-based Matlab backslash, and accurately
adapt to the size of the least squares residual.
for rectangular matrices, is based on a QR decomposition. 
This means, that the solution accuracy of the preconditioned normal equation depends on the residual
of the original least squares problem -- even though there is officially no least squares problem present and
the half-preconditioned normal equations
do not even have an equivalent least squares problem.

While the present paper focusses on perturbation theory and on numerical accuracy of direct methods, 
future work will investigate computational
speed. 
First is a comparison with established methods for solving least squares problems, including the QR decomposition, the unpreconditioned normal equations,
and the randomized iterative  solver Blendenpik \cite{Blendenpik}.
Second is a speed up of the preconditioned normal equations
via a mixed precision implementation, where  the preconditioner
 is computed in a lower arithmetic precision, and then promoted back to double precision for the 
computation of the preconditioned matrix.
Third is a GPU implementation of the mixed precision version.

\section*{Acknowledgement}
I thank James Garrison for helpful discussions, and two anonymous reviewers for help with 
improving the presentation.

\appendix
\section{Alternative perturbation bounds}\label{s_other}
We present a perturbation bound for the normal equations (Section~\ref{s_ne}), and
alternative perturbation bounds for the preconditioned normal equations (Section~\ref{s_apne}),
half-preconditioned normal equations (Section~\ref{s_ahpne}), and right preconditioned least squares problems.

\subsection{Perturbation of the normal equations}\label{s_ne}
Lemma~\ref{l_a2} presents a perturbation bound for the normal equations
that depends on the least squares residual.

Let $\ma\in\rmn$ with $\rank(\ma)=n$, and $\vb\in\real^m$. 
The exact normal equations are
\begin{equation*}
\ma^T\ma\vx_*=\ma^T\vb.
\end{equation*}
We perturb the matrix $\ma$ but make no assumptions on the size of the
perturbation, so that the perturbed matrix $\ma+\me$ has the potential to be 
rank deficient.

\begin{lemma}\label{l_a2}
Let $\me\in\rmn$, $\epsilon\equiv \|\me\|/\|\ma\|$, and
\begin{equation*}
(\ma+\me)^T(\ma+\me)\vhx=(\ma+\me)^T\vb.
\end{equation*}
If $\vhx\neq \vzero$, then
\begin{equation*}
\frac{\|\vx_*-\vhx\|}{\|\vhx\|}\leq \kappa(\ma)^2\>\epsilon\, 
\left(\frac{\|\vb-\ma\vhx\|}{\|\ma\|\|\vhx\|}+1+\epsilon\right).
\end{equation*}
\end{lemma}

\begin{proof}$\ $
Write the perturbed system as 
\begin{equation*}
\ma^T\vb-\ma^T\ma\vhx= \me^T(\ma\vhx-\vb)+(\ma+\me)^T\me\vhx.
\end{equation*}
Multiply by $(\ma^T\ma)^{-1}$
\begin{equation*}
\vx_*-\vhx=
(\ma^T\ma)^{-1}\left(\me^T(\ma\vhx-\vb)+(\ma+\me)^T\me\vhx\right),
\end{equation*}
and take norms.
\end{proof}

Lemma~\ref{l_a2} implies that, to first order, the relative error in $\vhx$ is bounded by
\begin{equation*}
\frac{\|\vx_*-\vhx\|}{\|\vhx\|}\lesssim \kappa(\ma)^2\,\epsilon\, 
\max\left\{\frac{\|\vb-\ma\vhx\|}{\|\ma\|\|\vhx\|}, 1\right\}.
\end{equation*}
This suggests that the solution accuracy of the normal equations
depends on the least squares residual when it is large, that is, if
$\frac{\|\vb-\ma\vhx\|}{\|\ma\|\|\vhx\|}>1$.   

\subsection{Alternative perturbation bounds for the preconditioned normal equations}\label{s_apne}
Lemmas \ref{l_a3} and~\ref{l_a6} present two alternative perturbation bounds for the
preconditioned normal equations, under the following assumptions.

Let $\ma\in\rmn$ have $\rank(\ma)=n$. Let $\mr_s\in\rnn$ be a fixed nonsingular matrix, and
$\ma_p\equiv \ma\mr_s^{-1}$.
The exact problem is
\begin{align*}
\ma_p^T\ma_p\vy_*&=\ma_p^T\vb\\
\mr_s\vx_*&=\vy_*.
\end{align*}
Since $\ma^T\ma$ and $\mr_s$ are nonsingular, so is the preconditioned matrix $\ma_p^T\ma_p$.

Lemma~\ref{l_a3} below perturbs all instances of the preconditioned matrix $\ma_p$ by the same 
perturbation~$\me_p$.

\begin{lemma}\label{l_a3}
Let $\me\in\rmn$, $\epsilon\equiv \|\me_p\|/\|\ma_p\|$, and
\begin{align}
(\ma_p+\me_p)^T(\ma_p+\me_p)\vhy&=(\ma_p+\me_p)^T\vb\label{e_l31}\\
\mr_s\vhx&=\vhy.\label{e_l32}
\end{align}
If $\vhy\neq \vzero$ and $\vhx\neq \vzero$, then
\begin{equation*}
\frac{\|\vx_*-\vhx\|}{\|\vhx\|}\leq \kappa(\mr_s)\ \nu\,\kappa(\ma_p)^2\>\epsilon\,
\left(\frac{\|\vb-\ma_p\vhy\|}{\|\ma_p\|\|\vhy\|}+1+\epsilon\right),
\qquad \nu\equiv \frac{\|\mr_s\vhx\|}{\|\mr_s\|\|\vhx\|}\leq 1.
\end{equation*}
\end{lemma}

\begin{proof}$\ $
Apply  Lemma~\ref{l_a2} to the system~(\ref{e_l31}),
\begin{equation*}
\frac{\|\vy_*-\vhy\|}{\|\vhy\|}\leq \kappa(\ma_p)^2\>\epsilon\,
\left(\frac{\|\vb-\ma_p\vhy\|}{\|\ma_p\|\|\vhy\|}+1+\epsilon\right),
\end{equation*}
and substitute the above into Lemma~\ref{l_4}.
\end{proof}

Lemma~\ref{l_a3} also holds for the solution from the perturbed Blendenpik problem
\begin{align*}
\|(\ma_p+\me_p)\vy_*-\vb\|_2&=\min_{\vy}{\|(\ma_p+\me_p)\vy-\vb\|_2}, \qquad 
\mr_s\vx=\vy.
\end{align*}
Lemma~\ref{l_a3} implies that, to first order, the relative error in $\vhx$ is bounded by
\begin{equation*}
\frac{\|\vx_*-\vhx\|}{\|\vhx\|}\lesssim \kappa(\mr_s)\ \kappa(\ma_p)^2\>\epsilon\,
\max\left\{1,\frac{\|\vb-\ma_p\vhy\|}{\|\ma_p\|\|\vhy\|}\right\}.
\end{equation*}
Numerical experiments indicate that this bound can be much smaller than the actual 
error.

The alternative bound in Lemma~\ref{l_a6} below perturbs all instances of the preconditioner 
$\mr_s$ by the same  matrix~$\me_s$.

\begin{lemma}\label{l_a6}
Let $\me_s\in\rmn$,  $\epsilon\equiv \tfrac{\|\me_s\|}{\|\mr_s\|}$, $\kappa(\mr_s)\epsilon<1$,
and
\begin{equation}
\widehat{\ma}_p\equiv \ma(\mr_s+\me_s)^{-1},\qquad
\widehat{\ma}_p^T\widehat{\ma}_p\vhy=\widehat{\ma}_p^T\vb,\qquad
\mr_s\vhx=\vhy.\label{e_t31}
\end{equation}
If $\vhy\neq \vzero$ and $\vhx\neq \vzero$, then
\begin{equation*}
\frac{\|\vx_*-\vhx\|}{\|\vhx\|}\leq \kappa(\mr_s)\,\nu\,\kappa(\ma_p)^2\, \eta\,\epsilon
\left(1+\frac{\|\vb-\ma_p\vhy\|}{\|\ma_p\|\|\vhy\|}+\eta\epsilon\right),
\end{equation*}
where 
\begin{equation*}
\nu\equiv \frac{\|\mr_s\vhx\|}{\|\mr_s\| \|\vhx\|}\leq 1\qquad \text{and}\qquad
\eta\equiv \frac{\kappa(\mr_s)}{1-\kappa(\mr_s)\epsilon}. 
\end{equation*}
\end{lemma}

\begin{proof}$\ $
We start by bounding the relative error in $\vhy$.
Write
\begin{equation*}
\widehat{\ma}_p=\ma(\mr_s+\me_s)^{-1}=\underbrace{\ma\mr_s^{-1}}_{\ma_p}
(\mi+\underbrace{\me_s\mr_s^{-1}}_{\mf})^{-1}=
\ma_p(\mi-\underbrace{(\mi+\mf)^{-1}\mf}_{\me_p})
=\ma_p(\mi-\me_p),
\end{equation*}
where 
$\mf\equiv \me_s\mr_s^{-1}$ and $\me_p\equiv (\mi+\mf)^{-1}\mf$.
Then (\ref{e_t31}) can be written as 
\begin{equation*}
\widehat{\ma}_p^T(\ma_p\vhy-\vb)=\widehat{\ma}_p^T\ma_p\me_p\vhy.
\end{equation*}
With $\widehat{\ma}_p=\ma_p(\mi-\me_p)$ this gives
\begin{equation*}
(\mi-\me_p)^T\ma_p^T(\vb-\ma_p\vhy)=(\mi-\me_p)^T\ma_p^T\ma_p\me_p\vhy.
\end{equation*}
Rearrange, 
\begin{equation*}
\ma_p^T(\vb-\ma_p\vhy)=\me_p^T\ma_p^T(\vb-\ma_p\vhy)+
(\mi-\me_p)^T\ma_p^T\ma_p\me_p\vhy.
\end{equation*}
Multiply by $(\ma_p^T\ma_p)^{-1}$, and let $\vy_*\equiv(\ma_p^T\ma_p)^{-1}\ma_p^T\vb$
 be the solution of (\ref{e_pne})
\begin{equation*}
\vy_*-\vhy=(\ma_p^T\ma_p)^{-1}\left(\me_p^T\ma_p^T(\vb-\ma_p\vhy)
+(\mi-\me_p)^T\ma_p^T\ma_p\me_p\vhy\right).
\end{equation*}
Take norms and use the fact that $\kappa(\ma_p^T\ma)=\kappa(\ma_p)^2$,
\begin{align*}
\frac{\|\vy_*-\vhy\|}{\|\vhy\|}&
\leq \kappa(\ma_p)^2\,\left(\|\me_p\|\,\frac{\|\vb-\ma_p\vhy\|}{\|\ma_p\|\|\vhy\|} +
(1+\|\me_p\|)\|\me_p\|\right)\\
&= \kappa(\ma_p)^2\,\|\me_p\|\,\left(\frac{\|\vb-\ma_p\vhy\|}{\|\ma_p\|\|\vhy\|} +
1+\|\me_p\|\right).
\end{align*}
At last bound
\begin{equation*}
\|\me_p\|\leq \frac{\|\mf\|}{1-\|\mf\|}\leq 
\frac{\|\me_s\| \|\mr_s^{-1}\|}{1-\|\me_s\| \|\mr_s^{-1}\|}
\leq \frac{\kappa(\mr_s)\epsilon}{1-\kappa(\mr_s)\epsilon}=\eta\,\epsilon,
\end{equation*}
so that 
\begin{equation*}
\frac{\|\vy_*-\vhy\|}{\|\vhy\|}\leq \kappa(\ma_p)^2\,\eta\,  \epsilon\,
\left(1+\frac{\|\vb-\ma_p\vhy\|}{\|\ma_p\|\|\vhy\|} +\eta\epsilon\right).
\end{equation*}
Substitute this bound into Lemma~\ref{l_4}.
\end{proof}

Lemma~\ref{l_a6} implies that, to first order, the error in $\vhx$ is bounded by
\begin{equation*}
\frac{\|\vx_*-\vhx\|}{\|\vhx\|}\lesssim \kappa(\mr_s)^2\,\kappa(\ma_p)^2\, \epsilon\,
\max\left\{1, \frac{\|\vb-\ma_p\vhy\|}{\|\ma_p\|\|\vhy\|}\right\}.
\end{equation*}
Numerical experiments indicate that this bound can be much larger than the actual 
error. 

\subsection{Alternative perturbation bound for the half preconditioned normal equations}\label{s_ahpne}
Lemma~\ref{l_a7} presents an alternative perturbation bound for the half preconditioned
normal equations, under the following assumptions.

Let $\ma\in\rmn$ have $\rank(\ma)=n$. Let $\mr_s\in\rnn$ be a fixed
nonsingular matrix and $\ma_p\equiv \ma\mr_s^{-1}$.
The exact problem is
\begin{equation*}
\ma_p^T\ma\vx_*=\ma_p^T\vb.
\end{equation*}
Since $\ma^T\ma$ and $\mr_s$ are nonsingular, so is the half preconditioned
matrix $\ma_p^T\ma$.

Lemma~\ref{l_a7} below perturbs both matrices by an additive perturbation.

\begin{lemma}\label{l_a7}
Let $\me_p, \me_A\in\rmn$,
$\epsilon\equiv \max\{\tfrac{\|\me_p\|}{\|\ma_p\|}, \tfrac{\|\me_A\|}{\|\ma\|}\}$, and
\begin{equation*}
(\ma_p+\me_p)^T(\ma+\me_A)\vhx=(\ma_p+\me_p)^T\vb.
\end{equation*}
If $\ma_p^T\ma$ is nonsingular and $\vhx\neq \vzero$, then
\begin{equation*}
\frac{\|\vx_*-\vhx\|}{\|\vhx\|}\leq \kappa(\ma_p^T\ma)\, \nu\, \epsilon\,
\left(\frac{\|\vb-\ma\vhx\|}{\|\ma\|\|\vhx\|}+1+\epsilon\right), \qquad 
\nu\equiv \frac{\|\ma_p\|\|\ma\|}{\|\ma_p^T\ma\|}\geq 1.
\end{equation*}
\end{lemma}

\begin{proof}$\ $
Multiply the perturbed system by $(\ma_p^T\ma)^{-1}$ and rearrange,
\begin{equation*}
\vx_*-\vhx=(\ma_p^T\ma)^{-1}\left(\me_p^T(\ma\vhx-\vb)+
(\ma_p+\me_p)^T\me_A\vhx\right).
\end{equation*}
Then take norms.
\end{proof}

Lemma~\ref{l_a7} implies that, to first order, the error in $\vhx$ is bounded by
\begin{equation*}
\frac{\|\vx_*-\vhx\|}{\|\vhx\|}\leq \kappa(\ma_p^T\ma)\, \nu\, \epsilon\, 
\max\left\{1, \frac{\|\vb-\ma\vhx\|}{\|\ma\|\|\vhx\|}\right\}.
\end{equation*}
Numerical experiments indicate that this bound can be much smaller than the actual 
error.

\subsection{Perturbation bounds for preconditioned least squares problems}\label{s_apls}
We derive two perturbation bounds for right preconditioned least squares solvers, 
in the spirit of Blendenpik \cite[Algorithm 1]{Blendenpik},
with the same type of perturbations as in Theorem~\ref{t_3}. The bound in Theorem~\ref{t_b1}
is better than the bound in Theorem~\ref{t_3},  and the bound in Theorem~\ref{t_b2} is worse.
Hence the comparison, in regard to conditioning, of the preconditioned normal equations versus the
right preconditioned least squares problem is inconclusive.

Right preconditioned least squares problems compute
\begin{equation*}\label{e_b}
\|\ma_p\vy_*-\vb\|_2=\min_{\vy}{\|\ma_p\vy-\vb\|_2}, \qquad
\mr_s\vx=\vy.
\end{equation*}

The perturbations below correspond to those for $\ma_2$ in (\ref{e_perturb}).

\begin{theorem}\label{t_b1}
Let $\me\in\rmn$, $\me_p\in\rmn$,
$\epsilon\equiv \|\me_p\|/\|\ma_p\|$, $\kappa(\ma_p)\epsilon<1$ and
\begin{align*}
\|(\ma_p+\me_p)\hat{\vy}-\vb\|=\min_{\vy}{\|(\ma_p+\me_p)\vy-\vb\|}, \qquad
\mr_s\vhx=\vhy.
\end{align*}
If $\vhy\neq \vzero$ and $\vhx\neq \vzero$, then
\begin{equation*}
\frac{\|\vx_*-\vhx\|}{\|\vhx\|}\leq \kappa(\mr_s)\ \nu\,\kappa(\ma_p)\>\epsilon\,
\left(1+\kappa(\ma_p)\frac{\|\vb-(\ma_p+\me_p)\vhy\|}{\|\ma_p\|\|\vhy\|}\right),
\qquad \nu\equiv \frac{\|\mr_s\vhx\|}{\|\mr_s\|\|\vhx\|}\leq 1.
\end{equation*}
\end{theorem}

\begin{proof}$\ $
This follows from Lemmas \ref{l_ls} and~\ref{l_4}.
\end{proof}

A comparison with Theorem~\ref{t_3} or (\ref{e_pne1a}) shows that  Theorem~\ref{t_b1}
is slightly better because it lacks the additional factor $\kappa(\mr_s)$ that amplifies the 
least squares residual.

The perturbations below correspond to those for $\ma_1$ in (\ref{e_perturb}).

\begin{theorem}\label{t_b2}
Let $\me_s\in\rmn$,  $\epsilon\equiv \|\me_s\|/\|\mr_s\|$, $\kappa(\mr_s)\epsilon<1$,
and
\begin{align*}
\|\ma(\mr_s+\me_s)^{-1}\hat{\vy}-\vb\|=\min_{\vy}{\|\ma(\mr_s+\me_s)^{-1}\vy-\vb\|}, \qquad
\mr_s\vhx=\vhy.
\end{align*}
If $\vhy\neq \vzero$ and $\vhx\neq \vzero$, then
\begin{equation*}
\frac{\|\vx_*-\vhx\|}{\|\vhx\|}\leq \kappa(\mr_s)\,\nu\,\kappa(\ma_p)^2\, \eta\,\epsilon
\left(1+\frac{\|\vb-\ma_p\vhy\|}{\|\ma_p\|\|\vhy\|}+\eta\epsilon\right),
\end{equation*}
where 
\begin{equation*}
\nu\equiv \frac{\|\mr_s\vhx\|}{\|\mr_s\| \|\vhx\|}\leq 1\qquad \text{and}\qquad
\eta\equiv \frac{\kappa(\mr_s)}{1-\kappa(\mr_s)\epsilon}. 
\end{equation*}
\end{theorem}

\begin{proof}$\ $
This follows from Lemma~\ref{l_a6}, since with 
$\widehat{\ma}_p\equiv \ma(\mr_s+\me_s)^{-1}$, the least squares problem 
$\min_{\vy}{\|\widehat{\ma}_p\vy-\vb\|}$ is mathematically equivalent to the system
$\widehat{\ma}_p^T\widehat{\ma}_p\vhy=\widehat{\ma}_p^T\vb$.
\end{proof}

Theorem~\ref{t_b2} implies that, to first order, the error in $\hat{\vx}$ is bounded by
\begin{equation*}
\frac{\|\vx_*-\vhx\|}{\|\vhx\|}\lesssim \kappa(\mr_s)^2\,\kappa(\ma_p)^2\, \epsilon\,
\max\left\{1, \frac{\|\vb-\ma_p\vhy\|}{\|\ma_p\|\|\vhy\|}\right\}.
\end{equation*}
Hence Theorem~\ref{t_b2} is worse than Theorem~\ref{t_3} or (\ref{e_pne1a}) 
because its bound is always proportional to $\kappa(\mr_s)^2\,\kappa(\ma_p)^2\, \epsilon$,
regardless of the size of the least squares residual.

\bibliography{HHbib}
\bibliographystyle{siam}

\end{document}